\documentclass{amsart}
\usepackage[margin=1.5cm]{geometry}
\numberwithin{equation}{section}
\usepackage{amssymb}
\usepackage{amsmath, amsfonts,amsthm,amssymb,amscd, verbatim,graphicx,color,multirow,booktabs, caption,tikz,tikz-cd, mathdots,bm}
\usepackage{tikz-cd}
 \usepackage{hyperref}
\usetikzlibrary{positioning}

\newtheorem{theorem}{Theorem}
\newtheorem{lemma}[theorem]{Lemma}

\newtheorem{proposition}[theorem]{Proposition}

   \newtheorem*{NNtheorem}{Theorem}
   
      \theoremstyle{definition}
     \newtheorem{definition}[theorem]{Definition}
     \newtheorem{example}[theorem]{Example}
     
     \theoremstyle{remark}
     \newtheorem{remark}[theorem]{Remark}

\newcommand{\Sym}{\mathop{\mathrm{Sym}}}

 \definecolor{mycolor}{rgb}{0.55,0.0,0.16}
  \definecolor{myred}{rgb}{0.6,0.0,0.16}
  \definecolor{mygreen}{rgb}{0.0,0.6,0.16}
  \definecolor{myviolet}{rgb}{1,0,1}

\hypersetup{
colorlinks=false,
allbordercolors=myred,
citebordercolor=mygreen
}

\makeatletter
\@namedef{subjclassname@2020}{%
  \textup{2020} Mathematics Subject Classification}
\makeatother

\begin{document}
\title[Random Schreier graphs and expanders]{Random Schreier graphs and expanders}
\author[L. Sabatini]{Luca Sabatini}
\address{Luca Sabatini}
%Dipartimento di Matematica  e Informatica ``Ulisse Dini'',\newline
% University of Firenze, Viale Morgagni 67/a, 50134 Firenze, Italy} 
\email{luca.sabatini@unifi.it}
	\maketitle

    \vspace{0.2cm}
\section{Random Schreier graphs}  \label{3sect}

The following proposition contains most of the proof of Theorem \ref{thMain}.
 We work as in \cite[Proof of Theorem 5]{2008ChrisMark}.
     
     \begin{proposition} \label{propMain}
     Let the finite group $G$ act transitively on $\Omega$, and $\varepsilon,\delta>0$.
   Then, for a random multiset $S \subseteq G$ of size
   $\left \lceil \frac{\log 4}{\varepsilon^2} \cdot \log \left( \frac{2|\Omega|}{\delta} \right) \right \rceil$,
   one has
   $$ Prob (\> \lambda(G \circlearrowleft \Omega,S \sqcup S^{-1}) \geq \varepsilon \> ) \> \leq \> \delta . $$
     \end{proposition}
      \begin{proof}
     Let $\varepsilon >0$,
     $S := \{ s_1,...,s_d \} \subseteq G$ be a random multiset of size $d \geq 1$, and $Y \leqslant G$ be the stabilizer of a point.
     Using the notation of Proposition \ref{propSpecSG},
     since the trivial representation has multiplicity one in a transitive permutation representation,
     $$  Prob( \> \lambda(G,Y,S \sqcup S^{-1}) \geq \varepsilon \> ) \hspace{0.5cm} = 
   \hspace{0.5cm} Prob \left( \> \left( \max_{1 \neq \pi \in Irr(G,Y)}
     \lambda({\bf M}_\pi) \right) \geq \varepsilon \> \right) , $$ 
     where $\lambda({\bf M}_\pi)$ is the largest absolute value of an eigenvalue of ${\bf M}_\pi$.
     By standard probability theory, the previous quantity is at most
     $\sum_{1 \neq \pi \in Irr(G,Y)} Prob ( \> \lambda({\bf M}_\pi) \geq \varepsilon \>) $.
     For every non-trivial $\pi \in Irr(G,Y)$ we have
     $$ {\bf M}_\pi = \frac{1}{2d} \sum_{i=1}^d \pi(s_i) + \pi(s_i^{-1}) = 
     \frac{1}{d} \sum_{i=1}^d {\bf X}_i , $$
   where ${\bf X}_i := \tfrac{1}{2} (\pi(s_i)+\pi(s_i^{-1}))$.
   Now the random matrices ${\bf X}_i$'s satisfy the hypotheses of Theorem \ref{thAW}.
   First, the eingevalues of ${\bf X}_i$ lie in $[-1,1]$, because the eigenvalues of $\pi(s_i)$ are roots of unity.
  Second, if $\rho$ is the regular representation of $G$, then $\sum_{g \in G} \rho(g)$ is the all-$1$ matrix,
   and it has rank $1$.
   By the decomposition of $\rho$, $\sum_{g \in G} \rho(g)$ is equivalent to a block-type matrix having the matrices $\sum_{g \in G} \pi(g)$
   as blocks.
   Since $\rho$ contains the trivial representation once,
   we have that $\sum_{g \in G} \pi(g)$ is the zero operator, for every non-trivial irreducible representation $\pi$.
    Hence
    \begin{align*}
    Prob(\lambda(G \circlearrowleft \Omega,S \sqcup S^{-1}) \geq \varepsilon) & \hspace{0.5cm} \leq \hspace{0.5cm}
      2 e^{- \tfrac{d \varepsilon^2}{\log 4} } \cdot
     \left( \sum_{1 \neq \pi \in Irr(G,Y)} \dim(\pi) \right) \hspace{0.5cm} \\ & \hspace{0.5cm} \leq \hspace{0.5cm}
     2 |\Omega| \cdot e^{- \tfrac{d \varepsilon^2}{\log 4} } , 
    \end{align*}
     and the proof follows.
     \end{proof}
     
     \begin{proof}[Proof of Theorem \ref{thMain}]
     Let $\lambda := \lambda(G \circlearrowleft \Omega,S \sqcup S^{-1})$,
     and fix $\varepsilon >0$.
     Let $\delta(|\Omega|)$ be a function tending to zero to be chosen later.
      Since $\lambda$ takes value in $[0,1]$, 
     applying Proposition \ref{propMain} with $\epsilon' := \varepsilon(1-\delta)$ and $\delta':= \delta \varepsilon$, we obtain
     $$ \mathbb{E}[\lambda] \> \leq \>
     1 \cdot Prob(\lambda \geq \varepsilon') +  \varepsilon' \cdot Prob(\lambda < \varepsilon') \leq
     \delta' + \varepsilon' \> = \> \varepsilon . $$
    With these parameters, $S$ is a random multiset of size
     $$ \left \lceil \frac{\log 4}{\varepsilon^2(1-\delta)^2} \cdot \log \left( \frac{2|\Omega|}{\delta \varepsilon} \right) \right \rceil . $$
   The desired bound on $|S|$ follows choosing a slow function for $\delta$, so that $\log (1/\delta(|\Omega|)) = o(\log|\Omega|)$.
  In this way we obtain
  $$ \log \left( \frac{2|\Omega|}{\delta \varepsilon} \right) = 
  \log 2 +\log|\Omega| +\log(1/\delta) + \log(1/\varepsilon) \leq \frac{2(1-\delta)^2}{\log 4} \log|\Omega| $$
  for every sufficiently large set $\Omega$,
  where the last inequality follows because $\delta \rightarrow 0$ when $|\Omega|$ grows, and $\log 4 < 2$.
     \end{proof}
     \vspace{0.1cm}
     
     We stress that, in the statements of Theorem \ref{thMain} and Proposition \ref{propMain},
     there is no substantial difference with the setting where repetitions are not allowed:
     in fact, $O(\log|\Omega|)$ randomly chosem elements in $G$ are almost surely made of distinct members.

     \vspace{0.2cm}
\section{Spectral gap and abelian sections} \label{4sect}

     Given a finite group $G$, the study of the algebraic conditions of $G$ which are necessary (or sufficient)
     for the existence of a Cayley graph on $G$ with good expansion has a long history,
     which goes back to \cite[Section 3]{1993LW}.
     However, that article does not involve the spectral gap directly (but the so-called {\itshape Kazhdan constant}),
     and, as we said, only concerns the Cayley graph case.
     With this in mind, it is easy to see that our Theorem \ref{thTheta} is equivalent to the following.
     
     \begin{proposition}[Improved Lubotzky-Weiss inequality] \label{propGLWI}
       Let $G$ act transitively on $\Omega$, and let $S \subseteq G$ be a symmetric multiset.
       Let $Y \leqslant G$ be the stabilizer of a point.
     Then, for every intermediate subgroup $Y \leqslant H \leqslant G$, one has
      $$ gap(G \circlearrowleft \Omega,S) \> \leq \> 5 \> |H/H'Y|^{-\frac{2}{|S||G:H|}} . $$
     \end{proposition}
     \vspace{0.1cm}
     
   The next lemma is the main new ingredient required for the proof of Proposition \ref{propGLWI}.
   It is a generalized and improved version of \cite[Proposition 3.9]{1993LW},
   and recalls the standard fact that $\overline{S}$, as defined in (\ref{eqRS}), generates $H$ if $S$ generates $G$.
     
          \begin{lemma} \label{lemSGSub}
     Let $Y \leqslant H \leqslant G$ and let $S \subseteq G$ be a symmetric multiset.
     If $\overline{S} \subseteq H$ is any multiset obtained via the Reidemeister-Schreier method, then
     $$ gap(H,Y,\overline{S}) \> \geq \> gap(G,Y,S) . $$
     Moreover
    $$ \lambda(H,Y,\overline{S}) \> \leq \> \lambda(G,Y,S) . $$
     \end{lemma}
     \begin{proof}
     Let $T \subseteq G$ be the transversal of $H$ which is used to provide $\overline{S}$,
     and let $R \subseteq H$ be any transversal of $Y$ in $H$.
 Moreover, let ${\bf M}$ and ${\bf M}_{_H}$ be the averaging operators of $Sch(G,Y,S)$ and $Sch(H,Y,\overline{S})$ respectively,
 and let ${\bf M}_{_H} {\bm f} = \tau {\bm f}$ for some eigenvalue $\tau \neq 1$ and eigenfunction
 ${\bm f} \in \ell^2(H/Y)$, where $H/Y$ denotes the set of the right-cosets of $Y$ in $H$.
 Let $\mathcal{Z}(H/Y) \subseteq \ell^2(H/Y)$ denote the (one-dimensional) subspace of constant functions.
 Since $\mathcal{Z}(H/Y)$ is the eigenspace of the trivial eigenvalue $1$, and ${\bf M}_{_H}$ is symmetric,
 we have ${\bm f} \in \mathcal{Z}(H/Y)^\perp$, i.e. ${\bm f}$ is a function which sums to zero.
 We extend ${\bm f}$ to $\widetilde{\bm f} \in \ell^{2}(G/Y)$ in the following way:
 $\widetilde f(Yrt) = f(Yr)$ for every $r \in R$ and $t \in T$.
 We notice that ${\widetilde {\bm  f}} \in \mathcal{Z}(G/Y)^\perp$, in fact
 $$ \langle {\widetilde{\bm f}},{\bf 1} \rangle = \sum_{r \in R \atop t \in T} \widetilde f(Yrt) =
 \sum_{r \in R \atop t \in T} f(Yr) = |G:H| \sum_{r \in R} f(Yr) = |G:H| \langle {\bm f},{\bf 1} \rangle = 0 . $$
 Moreover,
 $$ \langle \widetilde{\bm f} , \widetilde{\bm f} \rangle = 
 \sum_{r \in R \atop t \in T} \widetilde f(Yrt)^2 =
 \sum_{t \in T} \langle {\bm f}, {\bm f} \rangle 
 = |G:H| \langle {\bm f} , {\bm f} \rangle $$
 and 
 $$ \langle {\bf M} \widetilde{\bm f} , \widetilde{\bm f} \rangle =
 \frac{1}{|S|} \sum_{r \in R \atop t \in T} \sum_{s \in S} {\widetilde f}(Yrts) {\widetilde f}(Yrt) = 
 \frac{1}{|S|} \sum_{r \in R \atop t \in T} \sum_{s \in S} {\widetilde f}(Yrts) f(Yr) . $$
 Write $rts = rts (\overline{ts})^{-1} \overline{ts}$,
 so that ${\widetilde f}(Yrts) = f(Yrts (\overline{ts})^{-1})$.
 It follows that
$$ \langle {\bf M} \widetilde{\bm f} , \widetilde{\bm f} \rangle =
 \frac{|G:H|}{|\overline{S}|} \sum_{r \in R} \sum_{s \in \overline{S}} f(Yrs) f(Yr) =
|G:H| \langle {\bf M}_{_H} {\bm f} , {\bm f} \rangle . $$
 Finally, we have
 $$ \tau = \frac{\tau \langle {\bm f} , {\bm f} \rangle}{\langle {\bm f} , {\bm f} \rangle} =
 \frac{\langle \tau {\bm f} , {\bm f} \rangle}{\langle {\bm f} , {\bm f} \rangle} =
 \frac{\langle {\bf M}_{_H} {\bm f} , {\bm f} \rangle}{\langle {\bm f} , {\bm f} \rangle} =
 \frac{\langle {\bf M} \widetilde{\bm f} , \widetilde{\bm f} \rangle}{\langle \widetilde{\bm f} , \widetilde{\bm f} \rangle} . $$
Using the Rayleigh-Ritz Lemma \ref{1lemRRT}, we see that the last ratio lies in $[-1,1-gap(G,Y,S)]$.\\
Similarly, it lies in $[-\lambda(G,Y,S),\lambda(G,Y,S)]$.
     \end{proof}
     
     \begin{remark}
     When looking at $\overline{S}$ as a {\bfseries set} (i.e. without the multiplicities induced by the Reidemeister-Schreier method),
     the inequality $gap(H,Y,\overline{S}) \geq gap(G,Y,S)$ is not true in general.
     In fact, we found many counterexamples with the help of the GAP System \cite{2021GAP}:
     the easiest are Cayley graphs on the dihedral group of order $8$ and on the cyclic subgroup of order $4$.
     We remark that the Kazhdan constant \cite{1993LW} (as well as the diameter of a Schreier graph)
     does not distinguish between sets and multisets by definition.
     \end{remark}
     
     \begin{remark}
     Let $Y \leqslant H \leqslant G$, and $S \subseteq G$.
     When combined with Lemma \ref{lemBCSchreier},
     the second part of Lemma \ref{lemSGSub} implies that if $\overline{S}$ avoids some subgroup of $H$ of index $2$ which contains $Y$,
     then $S$ avoids some subgroup of $G$ of index $2$ which contains $Y$.
     In particular, if $G$ has no subgroups of index $2$,
     then {\bfseries all} Schreier graphs of the type $Sch(H,Y,\overline{S})$ are not bipartite.
     \end{remark}
     
     \begin{proof}[Proof of Proposition \ref{propGLWI}]
     From Lemma \ref{lemSGSub} we have $gap(G,Y,S) \leq gap(H,Y,\overline{S})$.
     Since $Spec(Sch(H,H'Y,\overline{S}))$ is contained in $Spec(Sch(H,Y,\overline{S}))$ from Proposition \ref{propSpecSG}, we gain
     $gap(H,Y,\overline{S}) \leq gap(H,H'Y,\overline{S})$.
     Since $H'Y$ is normal in $H$, from Remark \ref{remSchQuot} we have
     $Sch(H,H'Y,\overline{S}) \cong Cay(H/H'Y,\varphi(\overline{S}))$, where we look at $\varphi(\overline{S})$ with multiplicities.
      Now we can apply Lemma \ref{lemSGAb} to this last (abelian) Cayley graph, to obtain
     $$ gap(G,Y,S) \leq gap(H/H'Y,\varphi(\overline{S})) \leq 5 \> |H/H'Y|^{-2/|\varphi(\overline{S})|} . $$
     The fact that $|\varphi(\overline{S})|=|\overline{S}|=|G:H||S|$ concludes the proof.
     \end{proof}
     \vspace{0.1cm}
     
    We move to study nilpotent groups more in details.
     For every $i \geq 1$, let us denote by $\gamma_i(G)$ the $i$-th term of the lower central series of $G=\gamma_1(G)$.
     In particular, $\tfrac{\gamma_i(G)}{\gamma_{i+1}(G)}$ is the center of $\tfrac{G}{\gamma_{i+1}(G)}$.
 Moreover, we write $x^y:=y^{-1}xy$ and $[x,y] := x^{-1}x^y$.
      For every $x,y,z \in G$, from basic properties of commutators we have
      \begin{equation} \label{eqComm}
      [xy,z] \> = \> [x,z]^y [y,z] \hspace{1cm} \mbox{ and } \hspace{1cm} [x,yz] \> = \> [x,z] [x,y]^z .
      \end{equation}
      Finally, we remark that, if $x \in \gamma_i(G)$ and $y \in \gamma_j(G)$, then $[x,y] \in \gamma_{i+j}(G)$.
           
      \begin{proposition} \label{propDerNG}
     Let $G$ be a group, and let $Y \leqslant G$ be a subgroup of finite index
     such that $\langle Y,S \rangle=G$ for some set $S \subseteq G$ of size $d \geq 2$.
     Moreover, suppose that $\gamma_{c+1}(G) \subseteq Y$ for some $c \geq 1$.
     Then
    $$ |G:G'Y| \> \geq \> |G:Y|^{\beta(d,c)} , $$
    where $\beta(d,c):=\tfrac{d-1}{d^{c+1}-d^2+d-1} \geq \tfrac{1}{2d^c}$.
    \end{proposition}
    \begin{proof}
    For every $i \geq 1$ one has $\gamma_i(G)Y = \gamma_i(G) (\gamma_{i+1}(G)Y)$.
    Then $\gamma_{i+1}(G)Y \lhd \gamma_i(G)Y$, and the quotients $\gamma_i(G)Y/\gamma_{i+1}(G)Y$ are all abelian.
    Now we have
    $$ Y=\gamma_{c+1}(G)Y \lhd \gamma_{c}(G)Y \lhd ... \lhd \gamma_3(G)Y \lhd G'Y \lhd G . $$
    We write $x \equiv_i z$ to say that $x \gamma_i(G)Y=z \gamma_i(G)Y$.
    From (\ref{eqComm}), for every $i$, the commutator mapping defines a surjective bilinear form
	$$ \frac{\gamma_i(G)Y}{\gamma_{i+1}(G)Y} \times \frac{G}{G'Y}
	\twoheadrightarrow \frac{\gamma_{i+1}(G)Y}{\gamma_{i+2}(G)Y} \> \> , $$
	by setting $(x \gamma_{i+1}(G)Y,zG'Y) \rightarrow [x,z] \gamma_{i+2}(G)Y$ for every $x \in \gamma_i(G)$ and $z \in G$.
	To check that this is well-defined, let $\gamma_1 \in \gamma_{i+1}(G)$, $\gamma_2 \in G'$, $y_1,y_2 \in Y$.
	Working in $\tfrac{\gamma_{i+1}(G)Y}{\gamma_{i+2}(G)Y}$ we have
	 \begin{align*}
[x\gamma_1y_1,z\gamma_2y_2]  & \> \equiv_{i+2} \> 
[x,z\gamma_2] \\ & \> = \> 
 [x,\gamma_2] [x,z]^{\gamma_2}  \\ & \> \equiv_{i+2} \> 
  [x,z]^{\gamma_2} \\ & \> \equiv_{i+2} \> 
  [x,z] ,
\end{align*}
	because $[x, \gamma_2] \in \gamma_{i+2}(G)$ and $[x,z] \in \gamma_{i+1}(G)$.
	From the universal property of tensor product, there exists a surjective group homomorphism
	\begin{equation} \label{eqMap}
	\frac{\gamma_i(G)Y}{\gamma_{i+1}(G)Y} \otimes \frac{G}{G'Y}
	\twoheadrightarrow \frac{\gamma_{i+1}(G)Y}{\gamma_{i+2}(G)Y} 
	\end{equation}
	for every $1 \leq i \leq c-1$.
	If $d(G)$ denotes the minimal size of a generating set of $G$, we have $d(G/G'Y) \leq d$.
	The group on the left side of (\ref{eqMap}) can be generated by
	 $d(\gamma_i(G)Y/\gamma_{i+1}(G)Y) \cdot d$ elements.
	 It follows from (\ref{eqMap}) and induction that $d(\gamma_i(G)Y/\gamma_{i+1}(G)Y) \leq d^i$ for every $i \geq 1$.
	 Similarly, the order of an elementary element
	 $$ x \otimes z \in \frac{\gamma_i(G)Y}{\gamma_{i+1}(G)Y} \otimes \frac{G}{G'Y} $$
	 is the least common multiple between the orders of $x$ and $z$.
	By induction, the order of every element in $\gamma_i(G)Y/\gamma_{i+1}(G)Y$ is at most $|G/G'Y|$, for every $i \geq 1$.
	For all $i \geq 2$, this implies that
	$$ \left| \frac{\gamma_{i}(G)Y}{\gamma_{i+1}(G)Y} \right| \leq |G/G'Y|^{d^i} , $$
	while for $i=1$ we leave $|G/G'Y|$ in the computation.
	 Finally, we have
	 $$ |G:Y| = \prod_{i=1}^c |\gamma_i(G)Y/\gamma_{i+1}(G)Y| \leq |G/G'Y|^{1+ d^2 ...+d^c} =
	 |G/G'Y|^{\tfrac{d^{c+1}-1}{d-1}-d} . \qedhere $$
    \end{proof}
    \vspace{0.1cm}
    
    We remark that a similar inequality of the type $|G/G'| \geq |G|^{\varepsilon(d,c)}$ is given in \cite[Lemma 4.13]{2016BT}.
  Of course, the hypotheses of Proposition \ref{propDerNG} are verified when $G$ is a $d$-generated group of class $c$.
  The exponent $\beta(d,c)$ is not far from the best possible:
   the next example provides arbitrarily large groups where $|G/G'|$ is roughly $|G|^{c/d^{c-1}}$.
  
    \begin{example}
    Let $d \geq 3$ and define $W:=\{ x_1,...,x_d : (x_1)^2=...=(x_d)^2=1\}$.
    If $G_n:= W/\gamma_{n+1}(W)$ for every $n \geq 1$,
    then $G_n$ is a $d$-generated $2$-group of class $n$, and 
    $$ \log_2 |G_n/(G_n)'| = \log_2 |W/W'| = d . $$
    We are about to show that $G_n$ is quite large.
    Let $F$ be the free subgroup of $W$ which is generated by the $d-1$ elements $x_1 x_2$, $x_1 x_3$, ..., $x_1 x_d$.
    We have $|W:F|=2$ and $W \cong F \rtimes \langle x_1 \rangle$.
    Let $(P_n)_{n \geq 1}$ be the exponent-$2$ central series of $F=P_1$.
    By \cite[Lemma 6.2]{2017PSV} we have $\gamma_n(W)=P_n$ for every $n \geq 2$.
   Now the behavior of $(P_n)_{n \geq 1}$ is well known:
   by \cite[Lemma 20.7]{2007BNV} we obtain
    $$ \log_2 |G_n| = 1 + \log_2 |F:P_{n+1}| \sim \frac{(d-1)^{n+2}}{n(d-2)^2} $$
    when $n \rightarrow +\infty$. 
    \end{example}
    
    \begin{proof}[Proof of Theorem \ref{thSGNilp}]
   Let $Y \leqslant G$ be the stabilizer of a point.
    If $Y\langle S \rangle \neq G$, then there is nothing to prove, because $gap(G \circlearrowleft \Omega,S)=0$.
    Otherwise $S$ satisfies the hypotheses of Proposition \ref{propDerNG},
    and so $|G/G'Y| \geq |\Omega|^{\beta(|S|,c)}$.
   Applying Proposition \ref{propGLWI} with $H=G$ we obtain the first inequality in the statement.
   The second inequality follows arranging the terms.
    \end{proof}
     
     \begin{remark}
     A result similar to Theorem \ref{thSGNilp} can be obtained via a purely combinatorial method.
     Indeed, a nilpotent group of bounded class has ``polynomial growth'', as explained in \cite{1998B}.
    This property is inherited by the Schreier graph,
     and so it is easy to see that this graph has very large diameter.
     Then, it is well known that a large diameter implies a small spectral gap (see \cite[Sect. 2.4]{2006HLW}, for example).
     However, the bound we obtain along this road is weaker than Theorem \ref{thSGNilp} itself,
     and so it is not worth to state it precisely.
     \end{remark}

     \vspace{0.2cm}
     \section*{Acknowledgments} 
  The author thanks Pablo Spiga and the two anonymous referees for many useful comments and remarks.

%\vspace{0.2cm}
%\section*{Data availability}
%All data generated or analysed during this study are included in this published article.

\vspace{1cm}

\end{document}